\documentclass[reqno]{amsart}

\usepackage{latexsym,amsfonts}
\usepackage{amssymb}
\usepackage{amsmath,amsthm}          
\usepackage{euscript}                
\usepackage{epsfig}
\usepackage[all]{xy}
\usepackage{mathrsfs}
\usepackage{calrsfs}

\usepackage[bookmarks,
bookmarksnumbered,%
 colorlinks=true,%
 linkcolor=blue,%
 citecolor=red,%
 filecolor=blue,%
 urlcolor=blue,%
]
{hyperref}

\theoremstyle{plain} \numberwithin{equation}{section}
\newtheorem{main}{Theorem}

\newtheorem{thm}{Theorem}[section]

\newtheorem{prop}[thm]{Proposition}
\newtheorem{conj}[thm]{Conjecture}

\newtheorem{lemma}[thm]{Lemma}

\theoremstyle{definition}
\newtheorem{remark}{Remark}[section]

\newtheorem{defn}[remark]{Definition}

\newtheorem{rmk}[thm]{Remark}

\newcommand{\bi}{\begin{itemize}}
\newcommand{\ei}{\end{itemize}}
\newcommand{\bp}{\begin{proof}}
\newcommand{\ep}{\end{proof}}

\def\dim{\mbox{dim}}
\def\ra{\rightarrow}
\def\cal{\mathcal}

\def\CC{\mathbb{C}}
\def\PP{\mathbb{P}}

\def\RR{\mathbb{R}}

\def\EE{\cal E}

\def\BB{\cal B}

\def\GG{\cal G}

\def\s-{\setminus}

\begin{document}

\title{Balanced manifolds and $SKT$ metrics}

\author[Chiose]{Ionu\c{t} Chiose}

\address{
	Institute of Mathematics of the Romanian Academy,  P.O. Box 1-764, Bucharest, 014700,  Romania}
	
	\email{Ionut.Chiose@imar.ro}

\author[R\u asdeaconu]{Rare\c s R\u asdeaconu}

\address{        
        Department of Mathematics, 1326 Stevenson Center, Vanderbilt University, Nashville, TN, 37240, USA}
        
        \email{rares.rasdeaconu@vanderbilt.edu}

\author[\c Suvaina]{Ioana \c Suvaina}

\address{        
        Department of Mathematics, 1326 Stevenson Center, Vanderbilt University, Nashville, TN, 37240, USA}

\email{ioana.suvaina@vanderbilt.edu}




\begin{abstract} 
The equality between the balanced and the Gauduchon cones is discussed 
in several situations. In particular, it is shown that equality does not hold on 
many twistor spaces, and it holds on Moishezon manifolds. Moreover, it is 
proved that a  $SKT$ manifold of dimension three  on 
which the balanced cone equals the Gauduchon cone is in fact K\"ahler.
\end{abstract}

\maketitle
\tableofcontents

\section{Introduction}

Let $X$ be a closed complex  manifold of dimension $n$. Recall that a class in the 
Bott-Chern cohomology group $H^{1,1}_{BC}(X,\RR)$ is called pseudoeffective if it 
contains a closed positive current. The set of such classes form a closed convex cone 
in $ H^{1,1}_{BC}(X,\RR)$ called the pseudoeffective cone and it is denoted it by 
$\EE_{BC}^1.$ In  \cite {bdpp},   Boucksom, Demailly, P\u aun and Peternell introduced
the movable cone  $\cal M\subset H^{n-1,n-1}_{BC}(X)$. This cone is defined 
as the closure of the convex cone generated by classes of currents of the form
$p_*(\widetilde \omega_1\wedge \dots \wedge \widetilde\omega_{n-1}),$
where $p: \widetilde X\ra X$ is some modification and $\omega_i$ are K\"ahler forms on 
$\widetilde X$. 

The following remarkable conjecture  has been recently 
confirmed for projective manifolds by Witt Nystr\"om \cite{wittnystrom}, while the general 
case  is still open:
\begin{conj}[Conjecture 2.3, \cite{bdpp}]
\label{bdpp-conj}
For any K\"ahler manifold, 
$$
({\mathcal E}_{BC}^{1})^*=\overline{\cal M}.
$$  
\end{conj}
Extending the work of Toma \cite{toma} from projective to K\"ahler setting, it was observed 
by Fu and Xiao \cite[Theorem A.2]{fuxiao} (see also \cite[Remark 2.8]{crs}) that  Conjecture 
\ref{bdpp-conj} implies that for K\"ahler manifolds the movable cone is in fact the balanced cone 
$\BB$ of all positive $d$-closed smooth $(n-1,n-1)$-forms in $H^{n-1,n-1}_{BC}(X,\RR).$

\smallskip

A Hermitian metric $g$ on $X$ with co-closed K\"ahler form $\omega$ is called 
balanced. The class of balanced manifolds, i.e., the class of closed complex manifolds 
carrying balanced metrics, was introduced by Michelsohn \cite{michelsohn} who 
observed that prescribing a balanced metric (or equivalently its K\"ahler form) is the 
same as prescribing a positive $d$-closed smooth $(n-1,n-1)$-form. This class of 
manifolds has attracted considerable interest in the recent years. Most notably, 
Alessandrini and Bassanelli proved in \cite{alessandrini2} that unlike the class of 
K\"ahler manifolds, the class of balanced manifolds is closed under  bimeromorphisms. 
Furthermore, Fu, Li and Yau \cite{fly1} stressed  the importance of balanced manifolds 
from the perspective of heterotic string theory and constructed interesting non-K\"ahler 
examples in dimension three. Motivated by Conjecture \ref{bdpp-conj}, Fu and Xiao 
formulated the following:

\begin{conj}[Conjecture A.4., \cite{fuxiao}] 
\label{fx-conj}
For any compact balanced manifold 
$$
({\mathcal E}_{BC}^{1})^*=\overline{\cal B}.
$$  
\end{conj}

We give first many counter-examples to Conjecture \ref{fx-conj}. 
To formulate our result, recall that on a closed complex manifold $X$ 
one can define define the Gauduchon cone ${\mathcal G}$ as the 
set of all classes in the Aeppli cohomology group $H^{n-1,n-1}_A(X,\RR)$ 
which can be represented by a Gauduchon metric, i.e, by a  
$\partial\bar\partial$-closed positive $(n-1,n-1)$-form. Lamari's positivity 
criterion \cite[Lemme 3.3]{lamari} can be stated as 
$({\mathcal E}_{BC}^{1})^*=\overline\GG.$ 
Furthermore, let 
$$
\iota_{n-1}: H^{n-1,n-1}_{BC}(X,{\mathbb R})\to H^{n-1,n-1}_A(X,{\mathbb R})
$$ 
be the map induced by the identity.  
Since a balanced metric is also a Gauduchon metric, 
we have $\iota_{n-1}({\mathcal B})\subseteq {\mathcal G}.$  Therefore, the claim in 
Conjecture \ref{fx-conj} is  $\iota_{n-1}(\BB)= \GG,$ provided the ambient manifold 
is balanced.

\begin{main}
\label{main-bad-twistors}
There exists twistor spaces  such that $\iota_{n-1}(\BB)\subsetneqq \GG.$
\end{main}

Since the twistor spaces are known to carry balanced metrics \cite{michelsohn}, 
we obtain many counter-examples to the Fu-Xiao conjecture.

\bigskip 

On the other hand, based on the main result of Witt Nystr\"om 
\cite{wittnystrom}, we confirm the validity of Conjecture \ref{fx-conj} 
for Moishezon manifolds. For such manifolds, the 
$\partial\bar\partial$-lemma holds, and so  the map $\iota_{n-1}$ 
is an isomorphism. We prove:

\begin{main}
\label{B=GMoishezon}
For any Moishezon manifold $\BB=\GG.$
\end{main}

In particular, Theorems \ref{main-bad-twistors} and \ref{B=GMoishezon} are pieces 
of evidence in favor of a conjecture of Popovici \cite[Conjecture 6.1]{popovici}. 

\begin{conj}
\label{pop-conj}
If $X$ is a compact complex manifold on which the $\partial\bar\partial$-lemma holds, 
then $\BB=\GG.$
\end{conj}

Our last result is motivated by a conjecture of Fino and Vezzoni \cite{fino1}.
Recall  that a Hermitian metric $g$ with K\"ahler form $\omega$ on a compact 
complex manifold $X$ of dimension $n$ is called strongly K\"ahler with torsion 
($SKT$ for short) if $\omega$ is $\partial\bar\partial$-closed. It is known that a 
metric which is both balanced and $SKT$ is $d$-closed, hence K\"ahler \cite{bi}. 
Moreover, all the known examples of manifolds admitting a balanced metric 
and a $SKT$ metric are K\"ahler. For instance, in \cite{fly2} Fu, Li and Yau 
show that the examples of balanced non-K\"ahler manifolds they constructed 
do not carry $SKT$ metrics. Verbitsky \cite{verbitsky} showed that a twistor space  
with a $SKT$ metric is K\"ahler. In \cite{chiose1}, it is shown that a manifold in 
the Fujiki class  ${\mathcal C}$ (which is a balanced manifold by \cite{alessandrini1}) 
and which supports a $SKT$ metric is K\"ahler. In \cite{fino}, it is  proved that a 
nilmanifold which is balanced and $SKT$ is K\"ahler. It is therefore tempting to 
make the following conjecture

\begin{conj}
[Problem 3, \cite{fino1}]
\label{fv-conj}
A balanced and $SKT$ compact complex manifold is K\"ahler.
\end{conj} 
\noindent

We address here this conjecture for the class of complex manifolds satisfying 
$\iota_{n-1}({\mathcal B})= {\mathcal G}.$

\begin{main} 
\label{B+G=K}
Let $X$ be a compact complex  manifold of dimension three such that 
$\iota_{n-1}({\mathcal B})= {\mathcal G}.$
If $X$ carries a $SKT$ metric, then $X$ is K\"ahler.
\end{main}

\section{Preliminaries}

\begin{defn}
Let $(X,g)$ be a compact complex manifold of 
complex dimension $n$ equipped with a Hermitian
metric $g,$ and let $\omega$ denote its K\"ahler form. 
\begin{itemize}
\item[ i)] If $d(\omega^{n-1}) = 0,$ then $g$ is called a balanced 
metric. A complex manifold which admits a balanced 
metric is called a balanced manifold. 
\item[ ii)] If $\partial\bar\partial\omega = 0,$ then $g$ is called a strongly K\"ahler 
with torsion ($SKT$) metric. A complex manifold which admits a $SKT$
metric is called a $SKT$ manifold. 
\item[ iii)] If $d\omega = 0,$ then $g$ is called a K\"ahler metric. 
A complex manifold which admits a K\"ahler metric is 
called a K\"ahler manifold.
\end{itemize}
\end{defn}

Since the K\"ahler form of a Hermitian metric determines the metric, 
by an abuse of terminology we will not distinguish between the two notions. 
Moreover, according to Michelsohn \cite[page 279]{michelsohn}, 
given a positive $(n-1,n-1)$-form $\Phi$ on an $n$-dimensional manifold, 
there exists a positive $(1,1)$-form $\eta$ such that $\Phi=\eta^{n-1}.$ 
Therefore, prescribing a balanced or a Gauduchon metric is equivalent 
to prescribing a positive $(n-1,n-1)$-form which is $d$ or $\partial\bar\partial$-closed, 
respectively.

\subsection{Bott-Chern and Aeppli cohomologies and positive cones} 
Given a compact complex manifold $X$ of dimension $n$, we define 
the Bott-Chern cohomology groups
\begin{equation*}
H^{p,q}_{BC}(X,{\mathbb C})=
\frac{\{\alpha\in {\mathcal C}^{\infty}_{p,q}(X)\vert d\alpha=0\}}
{\{i\partial\bar\partial\beta\vert \beta\in {\mathcal C}^{\infty}_{p-1,q-1}(X)\}},
\end{equation*}
and the Aeppli cohomology groups
\begin{equation*}
H^{p,q}_A(X,{\mathbb C})=
\frac{\{\alpha\in {\mathcal C}^{\infty}_{p,q}(X)\vert i\partial\bar\partial \alpha=0\}}
{\{\partial\beta+\bar\partial\gamma\vert\beta\in {\mathcal C}^{\infty}_{p-1,q}(X),
\gamma\in {\mathcal C}^{\infty}_{p,q-1}(X)\}}
\end{equation*}

We use the notation $[s]$ for the class of a 
$d$-closed form or current $s$ in 
$H^{\bullet,\bullet}_{BC}$ and $\{t\}$ for the class of a 
$i\partial\bar\partial$-closed form or current $t$ in 
$H^{\bullet,\bullet}_A.$

The groups 
$H^{p,q}_{BC}(X, {\mathbb C})$ and 
$H^{n-p,n-q}_A(X, {\mathbb C})$ are dual via the pairing
\begin{equation*} 
H^{p,q}_{BC}(X, {\mathbb C})\times H_A^{n-p,n-q}
(X, {\mathbb C})\to {\mathbb C}, 
([\alpha ],\{\beta\})\to 
\int_X\alpha\wedge\beta.
\end{equation*}

Let $X$ be a compact complex manifold of dimension $n.$ 
The Gauduchon cone of $X$ is 
\begin{equation*}
\label{Gc}
\GG_X=\{\{\Omega\}\in H^{n-1,n-1}_{A}(X,\RR)| ~\Omega~\text{is a Gauduchon metric}\}
\end{equation*}
Similarly, we define the balanced cone:
\begin{equation*}
\label{Kbc}
\BB_X=\{[\Omega]\in H^{n-1,n-1}_{BC}(X,\RR)| ~
\Omega~\text{is a balanced metric}\}.
\end{equation*}

The Gauduchon cone is an open convex cone. According to Gauduchon 
\cite{gauduchon}, it is never empty. The balanced cone is open and convex. 
It can be empty, as there are examples compact complex manifolds which 
do not admit balanced metrics (e.g., see \cite{michelsohn}).

\medskip

The natural morphisms induced by the identity
\begin{equation*}
\iota_1:H^{1,1}_{BC}(X, {\mathbb R})\to H^{1,1}_A(X, {\mathbb R})
\end{equation*}
and
\begin{equation*}
\iota_{n-1}:H^{n-1,n-1}_{BC}(X,{\mathbb R})\to H^{n-1,n-1}_A(X, {\mathbb R})
\end{equation*} 
are well-defined, but in general, they are neither injective, nor surjective.  
They are however isomorphisms if $X$ is K\"ahler, or more generally 
on manifolds satisfying  the $\partial\bar\partial$-lemma. Nevertheless, we have
$$
\iota_{n-1}(\BB_X)\subseteq \GG_X,
$$ 
and so  $\iota_{n-1}(\overline \BB_X)\subseteq \overline \GG_X.$

\bigskip

For $\#\in\{BC, A\}$ and $p\in\{1,n-1\}$ we define the following cones:
\begin{enumerate}
\item the $\#-${\it pseudoeffective} cone 
\begin{equation*}
\label{pef}
{\mathcal E}^p_{X,\#}=\{\gamma\in H^{p,p}_{\#}
(X,{\mathbb R})\vert\exists T\geq 0, T\in\gamma\},
\end{equation*}
where by $T$ we denote here a current.   
\item the $\#-${\it nef} cone 
\begin{equation*}
\label{nef}
{\mathcal N}^p_{X,\#}=\{\gamma\in H^{p,p}_{\#}
(X,{\mathbb R})\vert \forall \varepsilon >0,
\exists\alpha_{\varepsilon}\in \gamma,
\alpha_{\varepsilon}\geq -\varepsilon \omega^p\},
\end{equation*}
where $\omega$ is the K\"ahler form of a fixed Hermitian 
metric on $X$ and $\alpha_{\varepsilon}$ denotes a smooth 
$(p,p)-$form. 
\end{enumerate}

\begin{rmk}
The pseudoeffective and nef cones ${\mathcal E}^1_{X,BC}$ 
and ${\mathcal N}^1_{X,BC}$ were first introduced by Demailly 
\cite[Definition 1.3]{de-reg}.
\end{rmk}

We recall next some of the properties and relations between the above cones.

\begin{prop}
Let $X$ be a compact complex manifold of dimension $n$. Then
\begin{itemize}

\item[i)] The cone ${\mathcal E}_{X,BC}^1$ is closed and 
${\mathcal N}_{X,BC}^1\subseteq {\mathcal E}_{X,BC}^1.$ 

\item[ii)] The cones ${\mathcal N}^p_{X,\#}$ are closed, where 
$p\in\{1,n-1\}$ and $\#\in \{BC, A\}$.  
\item[iii)] ${\mathcal N}_{X,A}^{n-1}=\overline\GG_X.$

\end{itemize}
Moreover, if $X$ is  balanced, then
\begin{itemize}
\item[iv)] ${\mathcal N}_{X,BC}^{n-1}=\overline{\cal B}_X.$
\item[v)] ${\mathcal E}_{X,A}^1$ is closed.

\end{itemize}

\end{prop}

\begin{proof}
For complete proofs we refer the interested reader to Lemmas 2.2, 2.3 
and 2.5 in \cite{crs}.
\end{proof}

We will often use the following result \cite[Theorem 2.4]{crs}  
(see also \cite[Remark 3.3]{fuxiao}),
which we state for the  convenience of the reader:

\begin{thm}
\label{duality}
Let $X$ be a compact complex manifold of dimension $n$. Then
\begin{itemize}

\item[i)]
${\mathcal N}_{X,BC}^1=({\mathcal E}_{X,A}^{n-1})^*,$

\item[ii)]
${\mathcal N}_{X,A}^{n-1}=({\mathcal E}_{X,BC}^1)^*.$
\end{itemize}
Moreover, if $X$ is  balanced, then
\begin{itemize}
\item[iii)]
${\mathcal N}_{X,A}^1=({\mathcal E}_{X,BC}^{n-1})^*,$

\item[iv)]
${\mathcal N}_{X,BC}^{n-1}=({\mathcal E}_{X,A}^1)^*.$
 
 \end{itemize}
 \end{thm}

We conclude this section with the following  result 
which indicates that the balanced cone is a natural 
generalization on balanced manifolds of the 
movable cone whose definition is confined to the 
Fujiki class $\cal C$ manifolds. 

\begin{prop}
\label{push-b-cone}
Let $\pi:X\to Y$ be a blow-up with smooth center between two balanced compact 
complex manifolds of dimension $n$.Then $\pi_*{\mathcal B}_X={\mathcal B}_Y.$ 
\end{prop}

\begin{proof}
The inclusion ${\mathcal B}_Y\subseteq \pi_*{\mathcal B}_X$ is just Corollary 4.9 
in \cite{alessandrini2}. Conversely, let $\omega^{n-1}$ be a balanced metric on $X$. 
From Theorem \ref{duality} iv), in order to show that the class $[\pi_*\omega^{n-1}]$ is 
balanced, it is enough to check that $([\pi_*\omega^{n-1}],\{T\})\geq 0$ where $T$ 
is an arbitrary $(1,1)$-current on $Y$ which is positive and $\partial\bar\partial$-closed. 
By  \cite{alessandrini2}, given such a current $T$,  there exists 
$\widetilde T$ a positive $(1,1)$-current on $X$ which is $\partial\bar\partial$-closed, 
such that $\pi_*\widetilde T=T$ and $\{\widetilde T\}=\pi^*\{T\}$. Then 
\begin{equation*}
([\pi_*\omega^{n-1}],\{T\})=([\omega^{n-1}],\{\widetilde T\})=\int_X\widetilde T\wedge \omega^{n-1}\geq 0
\end{equation*}
The above inequality is strict when the class $\{T\}\neq 0$, and this shows that 
$[\pi_*\omega^{n-1}]$ is in the interior of the cone ${\mathcal N}_{BC}^{n-1}$, 
which is the balanced cone.
\end{proof}

\begin{rmk}
In  \cite[Proposition 2.3]{xiao}, Xiao observed that one always has  
$\pi_*{\mathcal B}_X\subseteq{\mathcal B}_Y.$
\end{rmk}

\section{$\BB=\GG$ manifolds}

Let $X$ be a closed Hermitian manifold such that  $\iota_{n-1}(\BB_X)=\GG_X.$ 
Imposing such  condition has several implications on the complex structure of $X.$

\begin{lemma}
\label{iso-coh}
Let $X$ be a complex manifold such that $\iota_{n-1}(\BB_X)=\GG_X.$ Then 
$\iota_{n-1}$ is onto. If in addition $X$ is a SKT manifold, then 
$\iota_1$ and $\iota_{n-1}$ are isomorphisms.
\end{lemma}

\begin{proof}
Since ${\mathcal G}_X$ is open and non-empty, we see that $X$ is balanced and 
that $\iota_{n-1}$ is surjective. In particular, since $X$ is balanced, it follows that 
${\mathcal E}^1_A$ is closed. Since the cones ${\mathcal N}^{n-1}_{BC}$ and 
${\mathcal N}^{n-1}_A$ are the closures of the cones ${\mathcal B}_X$ and 
${\mathcal G}_X$ respectively, we get that 
\begin{equation}
\label{nefs-equality}
\iota_{n-1}({\mathcal N}_{BC}^{n-1})={\mathcal N}_A^{n-1}.
\end{equation} 
Dualizing (\ref{nefs-equality}), from Theorem \ref{duality} we obtain that 
$\iota_1({\mathcal E}_{BC}^1)={\mathcal E}_A^1$ and that $\iota_1$ is injective. 
Since $X$ is $SKT$, it follows that the interior of ${\mathcal E}_A^1$ is non-empty, 
therefore $\iota_1$ is also onto, hence an isomorphism. Therefore, $\iota_{n-1}$ is also an 
isomorphism.
\end{proof}

\subsection{Twistor spaces and counter-examples to the Fu-Xiao conjecture.} 
One can interpret Lemma \ref{iso-coh} as an obstruction  to the equality 
of the balanced and Gauduchon cones. We adopt this point of view and  
disprove  next Conjecture A.4 in \cite{fuxiao}. The counter-examples we 
propose are certain twistor spaces.  

\medskip

\subsubsection{Twistor spaces} Let $(M,g)$ be an oriented Riemannian $4-$manifold. 
The rank-6 vector bundle bundle of $2$-forms $\Lambda^2$ on $M$ 
decomposes as the direct sum of two rank-$3$ vector bundles
$$ \displaystyle 
\Lambda^2T^*M=\Lambda_+\oplus \Lambda_-
$$ 
By definition, $\Lambda^{\pm}$ are the eigenspaces of  the Hodge $\star$-operator
$$
\star:\Lambda^2T^*M\to \Lambda^2 T^*M,
$$  
corresponding to the $(\pm 1)-$ eigenvalues of $\star.$ 
The sections of $\Lambda^+$ are called  self-dual $2$-forms, whereas the 
sections of $\Lambda^-$ are the  anti-self-dual $2$-forms.

The Riemannian curvature tensor can be thought of as 
an operator
$$
\cal R: \Lambda^2 T^*M\rightarrow \Lambda^2 T^*M,
$$ 
called be the Riemannian
curvature operator. The Riemannian curvature operator 
decomposes under the action of $SO(4)$ as 
$$
\cal R=\frac{s}6Id+W^-+W^++\stackrel{\circ}{r},
$$ 
where  $W^{\pm}$ are trace-free 
endomorphisms of $\Lambda^{\pm},$ and they are called 
the self-dual and anti-self-dual components of the Weyl curvature 
operator. The  scalar curvature  $s$ acts by scalar multiplication 
and $\stackrel{\circ}{r}$ is the trace-free Ricci curvature operator.  

\begin{defn} 
An oriented Riemannian 
$4-$manifold $(M, g)$ is said to be anti-self-dual (ASD) 
if $W^+=0.$ 
\end{defn}

\begin{rmk} This definition is conformally invariant \cite{ahs}, i.e. 
if $(M,g)$ is ASD, so is $(M, ag)$  for any smooth positive 
function $a$. 
\end{rmk}

A plethora of anti-self-dual $4-$manifolds is rendered by a result of Taubes 
\cite{taubes} asserting that for any smooth, compact, oriented, 
$4$-dimensional manifold $X$, the connect sum $M=X\#k\overline {\CC\PP^2}$ 
of $X$ with $k$ copies of the complex projective $2$-space equipped with 
the opposite of its complex orientation admits a metric with $W^+=0$ for $k$ 
sufficiently large. In particular, one can find ASD manifolds with arbitrarily large 
first Betti number.

\bigskip

The twistor space of a conformal Riemannian manifold 
$(M,[g])$ is the total space of the sphere 
bundle of the rank three real vector bundle of self-dual 
$2-$forms 
$$
{\cal Z}:=S(\Lambda_+)\subset \Lambda^+.
$$ 
Atiyah, Hitchin, and Singer \cite{ahs} show that that 
$\cal Z$ comes naturally equipped with an almost complex structure, which is 
integrable if and only if $W^+ = 0.$ 

\medskip

In \cite{michelsohn}, Michelsohn states that the twistor space of a closed 
ASD manifold always carries a balanced metric, a result proved in \cite{muskarov} 
(see also \cite[Sect. 4]{crs}).

\bigskip

\subsubsection{Proof of Theorem  \ref{main-bad-twistors}} Let $\cal Z$ 
be the twistor space of a closed  anti-self-dual manifold $M$ 
of real dimension four. If ${\mathcal B}_{\cal Z}={\mathcal G}_{\cal Z}$ 
by Lemma \ref{iso-coh} 
\begin{equation*}
j_2:H^{2,2}_{BC}({\cal Z},{\mathbb C})\to H^{2,2}_A({\cal Z},{\mathbb C})
\end{equation*}
is surjective. Hence the natural morphism 
\begin{equation*}
\bar\partial:H_A^{2,2}({\cal Z},{\mathbb C})\to H_{BC}^{2,3}({\cal Z},{\mathbb C})
\end{equation*}
is zero. By duality, we obtain that 
 \begin{equation*}
 \bar\partial:H_A^{1,0}({\cal Z},{\mathbb C})\to H^{1,1}_{BC}({\cal Z},{\mathbb C})
 \end{equation*}
 is zero, which in turn implies that the natural morphism
 \begin{equation*}
 H^{1,0}_{\bar\partial}({\cal Z},{\mathbb C})\to H^{1,0}_A({\cal Z},{\mathbb C})
 \end{equation*}
 is surjective. Here $H^{\bullet,\bullet}_{\bar\partial}$ denotes the usual Dolbeault 
 cohomology. From \cite{singer} we know that 
 $H^{1,0}_{\bar\partial}(\cal Z,{\mathbb C})=0$, hence 
 $H^{1,0}_A(\cal Z,{\mathbb C})=H^{0,1}_A(\cal Z,{\mathbb C})=0$. 
 On the other hand, the morphism
 \begin{equation*}
 H^{0,1}_{\bar\partial}(\cal Z,{\mathbb C})\to H^{0,1}_A(\cal Z,{\mathbb C})
 \end{equation*}
 is always injective, therefore $H^{0,1}_{\bar\partial}(\cal Z,{\mathbb C})=0$. 
But, from Corollary 3.2 in \cite{singer}, it follows that 
$$
\dim H^{0,1}_{\bar\partial}(\cal Z,{\mathbb C})=\dim H_{dR}^1(M,{\mathbb C})
$$ 
where $H_{dR}^{\bullet}$ denotes the de Rham cohomology. 

Summing up, on a twistor space on which ${\mathcal B}_{\cal Z}={\mathcal G}_{\cal Z}$ 
one has $H^1_{dR}(M,{\mathbb C})=0.$ Therefore on the twistor spaces 
$X$ over the anti-self-dual $4$-folds $M$ with $H^1_{dR}(M,{\mathbb C})\neq 0$ 
the balanced cone cannot be equal to the Gauduchon cone. The existence of such 
anti-self-dual manifolds is ensured by the aforementioned theorem of Taubes \cite{taubes}.
\qed

\subsection{The balanced and Gauduchon cones on Moishezon manifolds}
The result of the previous section indicates that a generalization of Conjecture 2.3 in 
\cite{bdpp} to balanced manifolds fails. For projective manifolds,  the recent work of 
Witt Nystr\"om \cite{wittnystrom} implies  that  $\BB=\GG.$ We extend next Witt Nystr\"om's 
result to Moishezon manifolds.

\begin{prop}
\label{up-down}
If $\pi:X\to Y$ is a blow-up with smooth center in $Y$ and if 
$\iota_{n-1}({\mathcal B}_X)={\mathcal G}_X$, then $\iota_{n-1}({\mathcal B}_Y)={\mathcal G}_Y$.
\end{prop}
\begin{proof}
Since the Gauduchon cone on $X$ is never empty, it follows that the balanced cone on 
$X$ is non-empty, hence $X$ is balanced. Therefore $Y$ is balanced \cite{alessandrini1}. 
Consequently, ${\mathcal E}_{A,X}^1$ and ${\mathcal E}_{A,Y}^1$ are closed 
and the equality $\iota_{n-1}({\mathcal B}_X)={\mathcal G}_X$ is equivalent to 
${\mathcal E}_{A,X}^1={\mathcal E}_{BC, X}^1$ (\cite{crs} Theorem 2.4). So let 
$T$ be a positive current, $i\partial\bar\partial T=0$ on $Y$. 
Let $\pi^*T$ be its total transform on $X$ as defined in \cite{alessandrini2}. Since
${\mathcal E}_{A,X}^1={\mathcal E}_{BC,X}^1$, the class 
$\{\pi^*T\}\in H^{1,1}_A(X,{\mathbb R})$ contains a $d$-closed positive $(1,1)$-current $R$. 
Therefore the class of $T=\pi_*\pi^*T$ contains $\pi_*R$, a $d$-closed positive current. 
\end{proof}

As a consequence of Proposition \ref{up-down}, we have:

\begin{proof}[Proof of Theorem \ref{B=GMoishezon}] If $Y$ is projective,  
from \cite{wittnystrom}, as in the proof of  Proposition 2.10 in \cite{crs}
we have ${\mathcal B}_Y={\mathcal G}_Y.$
In general, a Moishezon manifold 
can be made projective by a sequence of blow-ups with smooth centers.
For each blow-up in the sequence we can apply Proposition \ref{up-down} and the 
conclusion follows.
\end{proof}

\begin{rmk}
 An interesting question is whether the condition ${\iota_{n-1}(\mathcal B})={\mathcal G}$ 
is a bimeromorphic invariant, i.e., given $X$ and $Y$ two bimeromorphic compact complex 
manifolds, is it true that $\iota_{n-1}({\mathcal B}_X)={\mathcal G}_X$ if and only if 
$\iota_{n-1}({\mathcal B}_Y)={\mathcal G}_Y?$ 
\end{rmk}

We conclude this section with the following observation which serves as an 
introduction to the next section.

\begin{prop}
Let $X$ be a compact complex surface. Then $\iota_{n-1}({\mathcal B}_X)={\mathcal G}_X$ 
if and only if $X$ is K\"ahler.
\end{prop}
\begin{proof}
If $\iota_{n-1}({\mathcal B}_X)={\mathcal G}_X$, since $\GG_X\neq \emptyset$ 
then ${\mathcal B}_X\neq \emptyset$, therefore $X$ is balanced.
Therefore $X$ is K\"ahler since a balanced metric on a surface is K\"ahler. Conversely, 
if $X$ is K\"ahler,  from Proposition 2.7 in \cite{crs} it follows that 
${\mathcal B}_X={\mathcal G}_X$.
\end{proof}

\section{$\BB=\GG$ on $SKT$ threefolds}

According to Popovici's Conjecture \ref{pop-conj},  every complex manifold on which  
the $\partial\bar\partial$-lemma holds, has the property $\BB=\GG.$ In particular, since 
the Gauduchon cone $\GG$ is open and non-empty, the manifold is balanced. 
While Conjecture \ref{pop-conj} is still open, there are known examples on which 
$\BB=\GG,$ and it's natural to consider the class of manifolds which satisfies that 
condition $\iota_{n-1}(\BB)=\GG.$ We address here Conjecture \ref{fv-conj} of 
Fino and Vezzoni on this class of balanced manifolds.

\smallskip

\begin{proof}[Proof of Theorem \ref{B+G=K}] Let $\eta$ be a $SKT$ metric on $X.$

\smallskip

{\it Step 1.} By Lemma \ref{iso-coh},  $\iota_1$ is an isomorphism.  
Hence, there exists a $(1,0)$ form $\alpha$ on $X$ such that 
$\gamma=\bar\partial\alpha+\eta+\partial\bar\alpha$ is a $d$-closed $(1,1)$ form. 
From $\iota_1({\mathcal E}^1_{BC})={\mathcal E}_A^1$ it follows that  
the class of 
$\gamma$ in $H^{1,1}_{BC}(X,{\mathbb R})$ is in ${\mathcal E}_{BC}^1.$ This means that 
there exists a $d$-closed positive $(1,1)$-current $T$ such that $[\gamma]=[T]$ in 
 $H^{1,1}_{BC}(X,{\mathbb R}).$

\medskip

{\it Step 2.} We will show that $[\gamma]$ is also in ${\mathcal N}_{BC}^1$, i.e., that it is nef. 
If the irreducible components of $\cup_{c>0}E_c(T)$ are all smooth, then we can use 
Th\'eor\`eme 2 in \cite{paun}. We have already checked that $[\gamma]$ 
is pseudoeffective, and let $Z$ be an irreducible analytic subset of $\cup_{c>0}E_c(T)$. 
If $Z$ is a curve, then 
$\int_Z\gamma=\int_Z\eta\geq 0$, hence $[\gamma]$ is nef on $Z$. If $Z$ is a surface, 
then (see Lemma 2.1 in \cite{chiose1}) 
\begin{equation}
\int_Z\gamma\wedge\gamma=\int_Z\eta\wedge\eta+2\int_Z\partial\alpha\wedge\bar\partial\alpha>0 
\end{equation}
and it is well known that in this case $Z$ is a K\"ahler surface. Let $\omega$ be a 
$d$-closed positive $(1,1)$ form on $Z$. Then clearly from Stokes' theorem we 
have $\int_Z\omega\wedge\gamma=\int_Z\omega\wedge\eta>0$ and this implies 
(see Theorem 4.5 (iii) in  \cite{demaillypaun}) that $[\gamma]$ is nef on $Z$. 
This implies that $[\gamma]$ is nef on $X$.

In general, fix $g$ a Hermitian metric on $X$ and let $\varepsilon>0$ be arbitrary. 
Then, from Theorem 3.2 in \cite{demaillypaun} it follows that there exists a closed 
current in the same class as $\gamma$, denoted 
$T_{\varepsilon}=\gamma+i\partial\bar\partial\varphi_{\varepsilon}\geq -\varepsilon g$ 
and $\pi_{\varepsilon}:X^{\varepsilon}\to X$ a sequence of blow-ups with smooth centers such that 
\begin{equation}
\pi_{\varepsilon}^*T_{\varepsilon}=\sum_i\lambda_i [D_i]+\omega_{\varepsilon}
\end{equation}
 where $D_i$ are smooth surfaces in $X^{\varepsilon}, \, \lambda_i>0$ and 
 $\omega_{\varepsilon}$ is a smooth $d$-closed $(1,1)$-form on $X^{\varepsilon}$.

Suppose
\begin{equation*}
X=X_0 \stackrel{\pi_1}{\longleftarrow} X_1\stackrel{\pi_2}{\longleftarrow}\dots 
\stackrel{\pi_{N-1}}{\longleftarrow} X_{N-1}\stackrel{\pi_N}{\longleftarrow}X_N=X^{\varepsilon}
\end{equation*}
is the sequence of blow-ups $\pi_{\varepsilon}:X^{\varepsilon}\to X$ and denote by $C_j$ 
the center of the blow-up $\pi_{j+1}:X_{j+1}\to X_j$ and by $E_j$ the exceptional divisor of 
the blow-up $\pi_j:X_j\to X_{j-1}$.

Now we construct by induction Hermitian metrics $g_j$ on $X_j$ as follows: 
set $g_0=g$ on $X_0=X$ and suppose that $g_j$ has been constructed on 
$X_j$. It is well-known that one can put a metric on the line bundle 
${\mathcal O}(-[E_{j+1}])$ on $X_{j+1}$ such that its curvature, 
denoted $\beta_{j+1}$, is supported in a small neighborhood of $E_{j+1}$, 
that $\beta_{j+1}$ is positive in a smaller neighborhood of $E_{j+1}$, and that, 
for $c_{j+1}$ a small enough non-negative constant, 
$g_{j+1}=\pi_{j+1}^*g_j+c_{j+1}\beta_{j+1}$ is a Hermitian metric on $X_{j+1}$. 
We choose $c_{j+1}$ such that
\begin{equation}\label{hermitian}
g_{j+1}\geq e^{-\frac{1}{(j+1)^2}}\pi_{j+1}^*g_j
\end{equation}
From (\ref{hermitian}) and from $T_{\varepsilon}\geq -\varepsilon g$ it follows that 
\begin{equation*}
T_N\geq -\varepsilon e^{\frac {1}{1^2}+\frac{1}{2^2}+...+\frac{1}{(N-1)^2}} 
g_N\geq -e^{\frac{\pi^2}{6}}\varepsilon g_N
\end{equation*}

On $X^{\varepsilon}=X_N$ we consider the current 
$T_N=\pi_{\varepsilon}^*T_{\varepsilon}$ which is in the 
same class as $\gamma_{\varepsilon}=\pi_{\varepsilon}^*\gamma$.
Each of the divisors $D_i$ are either proper transforms of divisors on 
$X$ or else components of the exceptional divisor of 
$\pi_{\varepsilon}:X^{\varepsilon}\to X$. In the first case, as above, we have 
\begin{equation*}
\int_{D_i}\gamma_{\varepsilon}\wedge\gamma_{\varepsilon}=
\int_Z\eta_{\varepsilon}\wedge\eta_{\varepsilon}+
2\int_Z\partial\alpha_{\varepsilon}\wedge\bar\partial\alpha_{\varepsilon}>0 
\end{equation*}
where $\eta_{\varepsilon}=\pi_{\varepsilon}^*\eta$ and $\alpha_{\varepsilon}=
\pi_{\varepsilon}^*\alpha$, hence $D_i$ is a K\"ahler surface with a K\"ahler form 
$\omega_i$, and from 
\begin{equation*}
\int_C\gamma_{\varepsilon}\geq 0, \int_{D_i}\gamma_{\varepsilon}\wedge \omega_i\geq 0
\end{equation*}
where $C$ is some curve in $D_i$, it follows that $[\gamma_{\varepsilon}]$ 
is nef on $D_i$. In the second case, $D_i$ is the projectivization of a rank $2$ 
vector bundle on a curve, hence it is K\"ahler, and again, from the inequalities
\begin{equation*}
\int_C\gamma_{\varepsilon}=\int_C\eta_{\varepsilon}\geq 0, \int_{D_i}
\gamma_{\varepsilon}\wedge \omega_i=\int_{D_i}\eta_{\varepsilon}\wedge\omega_i\geq 0
\end{equation*}
it follows that $[\gamma_{\varepsilon}]$ is nef on $D_i$.

Therefore, the current $T_N=\gamma_{\varepsilon}+
i\partial\bar\partial\pi_{\varepsilon}^*\varphi_{\varepsilon}$ 
satisfies $T_N\geq -e^{\frac{\pi^2}{6}}\varepsilon g_N$ and its Bott-Chern class 
$[\gamma_{\varepsilon}]$ is nef on each $D_i$. By using the techniques in  
\cite{paun}, we can find a ${\mathcal C}^{\infty}$ function $\psi_{N}$ on $X_{\varepsilon}$ such that
\begin{equation*}
\gamma_{\varepsilon}+i\partial\bar\partial\psi_N\geq -2 e^{\frac{\pi^2}{6}}\varepsilon g_N
\end{equation*}
Namely, from Lemme 1 in \cite{paun}, we first find a smooth function $f_{\varepsilon}$ in a 
neighborhood of $\cup_iD_i$ such that 
$\gamma_{\varepsilon}+i\partial\bar\partial f_{\varepsilon}\geq -e^{\frac{\pi^2}{6}}\varepsilon g_N$ 
and then take for $\psi_N$ a regularization of the maximum between $f_{\varepsilon}-C$ and 
$\pi_{\varepsilon}^*\varphi_{\varepsilon}$, where $C$ is some large constant. Note that 
$\pi^*\varphi_{\varepsilon}$ takes the value $-\infty$ on $\cup_iD_i$.

Now we construct by induction ${\mathcal C}^{\infty}$ functions 
$\psi_{\varepsilon,j}$ on $X_j$ such that 
\begin{equation}\label{inequality}
\gamma_j+i\partial\bar\partial\psi_{\varepsilon,j}\geq 
-2e^{\frac{\pi^2}{6}}\left (\frac {1}{2^0}+\frac {1}{2^1}+...+\frac{1}{2^{N-j}}\right ) \varepsilon g_j
\end{equation}
where $\gamma_j$ is the pull-back of $\gamma$ to $X_j$.

Set $\psi_{\varepsilon,N}=\psi_{N}$ and suppose that $\psi_{\varepsilon,j+1}$ 
has been constructed on $X_{j+1}$ such that

\begin{equation}\label{localfunction}
\gamma_{j+1}+i\partial\bar\partial \psi_{\varepsilon,j+1}\geq 
- 2e^{\frac{\pi^2}{6}}\left (\frac {1}{2^0}+\frac {1}{2^1}+...+\frac{1}{2^{N-j-1}}\right ) \varepsilon g_{j+1}
\end{equation}

 If $C_j$ is the center of the blow-up $\pi_{j+1}:X_{j+1}\to X_j$, then 
 $\gamma_j\vert C_j$ is nef since $C_j$ is a curve or a point and if $C_j$ is a curve, 
 $\int_{C_j}\gamma_j\geq 0$, so from Lemme 1 in \cite{paun}, there exists $U_j$ 
 a neighborhood of $C_j$ and $\lambda_j$ a ${\mathcal C}^{\infty}$ function on $U_j$ such that
 
 \begin{equation*}
 \gamma_j+i\partial\bar\partial\lambda_j\geq - e^{\frac{\pi^2}{6}}
 \left (1+\frac 12+...+\frac {1}{2^{N-j-1}}+\frac{1}{2^{N-j+1}}\right )\varepsilon g_{j}
 \end{equation*}

Pushing forward (\ref{localfunction}) to $X_j$ we obtain

\begin{equation*}
\gamma_j+i\partial\bar\partial\pi_{j+1,*}\psi_{\varepsilon ,j+1}\geq  
- 2e^{\frac{\pi^2}{6}}\left (1+\frac {1}{2}+...+\frac{1}{2^{N-j-1}}\right ) 
\varepsilon (g_j+c_{j+1}\pi_{j+1,*}\beta_{j+1} )
\end{equation*}
It is well-known that $\pi_{j+1,*}\beta_{j+1}$ is $i\partial\bar\partial$ exact 
(since $\beta_{j+1}$ is in the same class as $-[E_{j+1}]$
and the push forward of $[E_{j+1}]$ is $0$), so let 
$\pi_{j+1,*}\beta_{j+1}=i\partial\bar\partial\mu_j$. Therefore
\begin{equation*}
\gamma_j+i\partial\bar\partial \left [\pi_{j+1,*}\psi_{\varepsilon,j+1}+
2e^{\frac{\pi^2}{6}}\left ( 1+\frac 12+...+\frac {1}{2^{N-j-1}}\right )\varepsilon c_{j+1}\mu_j \right ]\geq
\end{equation*}
\begin{equation*}
-2e^{\frac{\pi^2}{6}}\left ( 1+\frac 12+...+\frac {1}{2^{N-j-1}}\right )\varepsilon g_j
\end{equation*}
and, as in \cite{paun}, using the function $\lambda_j$ constructed above, 
and a function $\zeta$ as in Lemme 2 in \cite{paun}, we obtain a ${\mathcal C}^{\infty}$ 
function $\psi_{\varepsilon,j}$ on $X_j$ which satisfies (\ref{inequality}).

For $j=0$ we obtain a ${\mathcal C}^{\infty}$ function $\psi_{\varepsilon,0}$ such that
\begin{equation*}
\gamma+i\partial\bar\partial\psi_{\varepsilon,0}\geq - 4e^{\frac{\pi^2}{6}}\varepsilon g
\end{equation*}
This means that $[\gamma]$ is nef.
\medskip

{\it Step 3.} From Lemma 2.1 in \cite{chiose1} with $k=3$ we obtain 
\begin{equation*}
\int_X\gamma^3=\int_X\eta^3+6\int_X\eta\wedge\partial\alpha\wedge\bar\partial\bar\alpha>0
\end{equation*}
and we can use Theorem 4.1 in \cite{chiose2}. Indeed, $[\gamma]$ is a nef class, 
of positive self-intersection, and $X$ is a $3$-fold that supports a $SKT$ metric 
(see Remark 4.3 in \cite{chiose2}). This implies that $X$ is K\"ahler.
\end{proof}

\begin{rmk}
Given $Z$ a singular component of $\cup_{c>0}E_c{T}$, it is clear that, if 
$p:\widetilde X\to X$ is a resolution of singularities of $Z$, then $\widetilde Z$ 
(the proper transform of $Z$) is K\"ahler and that $p^*\gamma$ is nef on 
$\widetilde Z$. This implies that, for every $\varepsilon>0$, there exists 
$\widetilde\varphi_{\varepsilon}$ a ${\mathcal C}^{\infty}$ function on 
$\widetilde Z$ such that $p^*\gamma+i\partial\bar\partial\widetilde\varphi_{\varepsilon}\geq 
- \varepsilon \widetilde g$  on $\widetilde Z$. However, it is not obvious to us that from 
this data one can construct a ${\mathcal C}^{\infty}$ function $\varphi_{\varepsilon}$ 
on $Z$ which satisfies $\gamma+i\partial\bar\partial\varphi_{\varepsilon}\geq -\varepsilon g$ 
(so that it is as in D\'efinition 3 in \cite{paun}) since such a function has to be locally
 the restriction of a ${\mathcal C}^{\infty}$ function on a neighborhood of $Z$. This is the 
 reason for having to consider the approximation current $T_{\varepsilon}$ in the proof 
 of Theorem \ref{B+G=K} above instead of working directly with the closed positive current $T$.
\end{rmk}

\begin{rmk}
In \cite{verbitsky} Verbitsky showed that a twistor space which supports a $SKT$ 
metric is K\"ahler. Our result does not imply Verbitsky's result. On a twistor space, 
it is not  always true that the balanced cone equals the Gauduchon cone, as noticed 
in the proof of Theorem \ref{main-bad-twistors}.
\end{rmk}

\subsection*{Acknowledgements}  The first author was supported by 
 a CNCS - UEFISCDI grant, project no.
PN-III-P4-ID-PCE-2016-0341.
The second author acknowledges the support of the Simons Foundation's 
"Collaboration Grant for Mathematicians", while the third author was supported 
by the NSF grant DMS-1309029. 

 \providecommand{\bysame}{\leavevmode\hbox
to3em{\hrulefill}\thinspace}

\end{document}